\documentclass[11pt]{amsart}

\usepackage[pdftex]{graphicx}
\usepackage[dvipsnames]{xcolor}

\usepackage[margin=3cm]{geometry}

\usepackage{amsfonts}
\usepackage{amsmath}
\usepackage{amssymb}
\usepackage{amsthm}
\usepackage[all]{xy}

\usepackage[english]{babel} 
\usepackage[utf8]{inputenc}

\usepackage{paralist}
\usepackage[colorlinks,cite color=blue,pagebackref=true,pdftex]{hyperref}

\usepackage[T1]{fontenc}

\newcommand\Defn[1]{\textbf{#1}}

\renewcommand\emptyset{\varnothing}
\newcommand\Z{\mathbb{Z}}               

\newcommand\R{\mathbb{R}}

\newcommand\eps{\varepsilon}

\newcommand\Sphere{\mathbf{S}}%

\newcommand\Av{\mathcal{A}} %
\newcommand\Avf[1]{\Av(#1)}
\newcommand\Rain{\mathcal{K}}

\newcommand\betti{\tilde{\beta}}

\newcommand\cS{\mathcal{S}}

\newcommand\n{\mathbf{n}}
\newcommand\Oct{\mathsf{O}}
\newcommand\Zil{\mathsf{Z}}

\newcommand\join{\ast}

\newcommand\s{s}%
\newcommand\f{C}
\newcommand\D{\mathcal{V}}
\newcommand\U{\mathcal{U}}

\DeclareMathOperator*{\csd}{depth}%
\DeclareMathOperator*{\Symm}{\triangle}
\DeclareMathOperator*{\relint}{relint}
\DeclareMathOperator*{\Int}{int}
\DeclareMathOperator*{\aff}{aff}
\DeclareMathOperator*{\conv}{conv}

\DeclareMathOperator*{\im}{im}

\newtheorem{thm}{Theorem}[section]
\newtheorem{cor}[thm]{Corollary}
\newtheorem{lem}[thm]{Lemma}
\newtheorem{prop}[thm]{Proposition}
\newtheorem{conj}{Conjecture}
\newtheorem*{Conj}{Conjecture}
\newtheorem{quest}{Question}

\theoremstyle{definition}

\newtheorem{example}[thm]{Example}

\title{Colorful simplicial depth, Minkowski sums, and generalized Gale
transforms}

\author[Adiprasito]{Karim Adiprasito}
\address{Einstein Institute for Mathematics, %
Hebrew University of Jerusalem, %
Jerusalem, Israel}
\email{adiprasito@math.huji.ac.il}

\author[Brinkmann]{Philip Brinkmann}
\address{Fachbereich Mathematik und Informatik, %
Freie Universit\"at Berlin, Berlin, %
Germany}
\email{pbrinkmann@math.fu-berlin.de}

\author[Padrol]{Arnau Padrol}
\address{
Sorbonne Universit\'es, Universit\'e Pierre et Marie Curie (Paris 6), Institut de Math\'ematiques de Jussieu - Paris Rive Gauche (UMR 7586), Paris, France}
\email{arnau.padrol@imj-prg.fr}

\author[Pat\'ak]{Pavel Pat\'ak}
\address{Einstein Institute for Mathematics, %
Hebrew University of Jerusalem, %
Jerusalem, Israel}
\email{ppatak@seznam.cz}

\author[Pat\'akov\'a]{Zuzana Pat\'akov\'a}
\address{Einstein Institute for Mathematics, %
Hebrew University of Jerusalem, %
Jerusalem, Israel}
\email{zuzka@kam.mff.cuni.cz}

\author[Sanyal]{Raman Sanyal}
\address{Fachbereich Mathematik und Informatik, %
Freie Universit\"at Berlin, Berlin, %
Germany}
\email{sanyal@math.fu-berlin.de}


\date{\today}

\begin{document}

\begin{abstract}
    The colorful simplicial depth of a collection of $d+1$ finite sets of
    points  in Euclidean $d$-space is the number of choices of a point from
    each set such that the origin is contained in their convex hull.  We use
    methods from combinatorial topology to prove a tight upper bound on the
    colorful simplicial depth. This implies a conjecture of Deza \emph{et
    al.}\ (2006).  Furthermore, we introduce colorful Gale transforms as a
    bridge between colorful configurations and Minkowski sums. Our colorful
    upper bound then yields a tight upper bound on the number of totally mixed
    facets of certain Minkowski sums of simplices. This resolves a conjecture
    of Burton (2003) in the theory of normal surfaces.
\end{abstract}

\maketitle

\section{Introduction}\label{sec:intro}

The classical Carath\'eodory theorem states that for any point $p$ in the
convex hull of a finite set $V \subset \R^d$, there is a subset $S \subseteq
V$ of at most $d+1$ points with $p \in \conv(S)$. Of course, this is is
invariant under translation and for all practical purposes one can always
assume $p = 0$.  In 1982, Imre B\'ar\'any~\cite{bar82} gave a beautiful
\emph{colored} version of this result.

\begin{thm}[Colorful Carath\'eodory Theorem]\label{thm:CC}
    Let $C_0,\dots,C_d \subset \R^d$ be finite sets such that $0$ is in the
    relative interior of $\conv(C_i)$ for all $0 \le i \le d$. Then there is
    some $S \subseteq C_0 \cup \cdots \cup C_d$ with $|S \cap C_i| = 1$ for
    all $i$ such that $0 \in \conv(S)$.
\end{thm}

The sets $C_i$ are thought of as color classes and $C = \{ C_0,\dots, C_d \}$
is called a \Defn{colorful configuration}. We call $C$ \Defn{centered} if $0
\in \relint \conv(C_i)$ for all $0 \le i \le d$. The assumption that $C$ is
centered can be considerably weakened (see, for
example,~\cite{ArochaBaranyBrachoFabilaMontejano2009, bar82,
HolmsenPachTverberg2008, MeunierDeza2013}) but in this paper we will only
consider centered configurations; see Question~\ref{quest:weaker}. We call a
subset $S \subseteq \bigcup_i C_i$ a \Defn{colorful simplex} if $|S \cap C_i|
\le 1$ for all $i$ and we call the simplex $S$ \Defn{hitting} if $\dim S = |S|
- 1 = d$ and $0 \in \conv(S)$.

Inspired by the simplicial depth introduced in~\cite{Liu}, Deza, Huang,
Stephen, Terlaky~\cite{DezaEtal2006} introduced the \Defn{colorful simplicial
depth}, $\csd(C)$, of a colorful configuration $C$ as the number of hitting
simplices of $C$. In particular, they initiated a systematic study of the
extremal values of $\csd(C)$ as $C$ ranges over all colorful configurations in
$\R^d$. The Colorful Carath\'eodory Theorem asserts that for any centered
configuration, there is at least one hitting simplex. For lower bounds on the
colorful simplicial depth, Deza \emph{et al.} confine themselves to the case
of colorful configurations in $\R^d$ such that the \Defn{core} $\bigcap_i
\conv(C_i)$ is of full-dimension $d$ and contains the origin in its interior. 
In particular, they conjectured that if
$|C_i| = d+1$, then 
\[
    1 + d^2 \ \le \ \csd(C).
\]
The initial lower bound of $2d$ from~\cite{DezaEtal2006} was improved in a
series of papers~\cite{BaranyMatousek2007,DezaMeunierSarrabezolles2014,
DezaStephenXie2011, StephenTomas2008} culminating in the resolution of the
conjectured lower bound by Sarrabezolles~\cite{Sarrabezolles2015}.
In~\cite{DezaEtal2006} also a conjectured upper bound was proposed.

\begin{conj}[{\cite[Conj.~4.4.]{DezaEtal2006}}]
    Let $C = \{ C_0, \dots, C_d \}$ be a centered colorful configuration in
    $\R^d$ with $|C_i| = d+1$ for all $0 \le i \le d$ and $0$ in the interior 
    of the core. Then
    \[
        \csd(C) \ \le \ 1 + d^{d+1}.
    \]
\end{conj}

\subsection{Upper Bounds on colorful simplicial depth}
The first goal of this paper is to prove this conjecture in the following
stronger form.  We say that a colorful configuration $C$ in $\R^d$ is in
\Defn{relative general position} if no colorful simplex $S$ of $C$ of
dimension $d-1$ contains the origin in its convex hull.
\begin{thm}\label{thm:CC_upper}
    Let $C = \{ C_0,\dots,C_d \}$ be a centered colorful configuration in
    relative general position in $\R^d$ with $|C_i| \ge 2$ for all $0 \le i
    \le d$.  Then
    \[
        \csd(C) \ \le \ 1 + \prod_{i=0}^d (|C_i| - 1).
    \]
\end{thm}

The \emph{relative general position} assumption is natural as otherwise for a
colorful configuration with $C_1 = \{0\}$, we would have the trivial upper
bound of $\prod_i |C_i|$.  
In general, even if the origin is contained in the interior of the 
full-dimensional core, adding several copies of $0$ to the color classes
produces configurations with a much larger colorful simplicial depth than the expected bound.
This behavior can happen even if all classes avoid the origin; for example, by adding
several copies of a point $p$ to $C_0$ and several of $-p$ to $C_1$.
To avoid these artifacts that arise from non-generic configurations, one could alternatively require hitting simplices to contain the origin in its
interior. The upper bound is then also attained at colorful configurations in
relative general position and our result applies.  The requirement that $|C_i|
\ge 2$ is subsumed by centered and relative general position and is therefore
redundant.

Let us first note that the bound of Theorem~\ref{thm:CC_upper} is tight.

\begin{figure}[htpb]
\includegraphics[width=.45\linewidth]{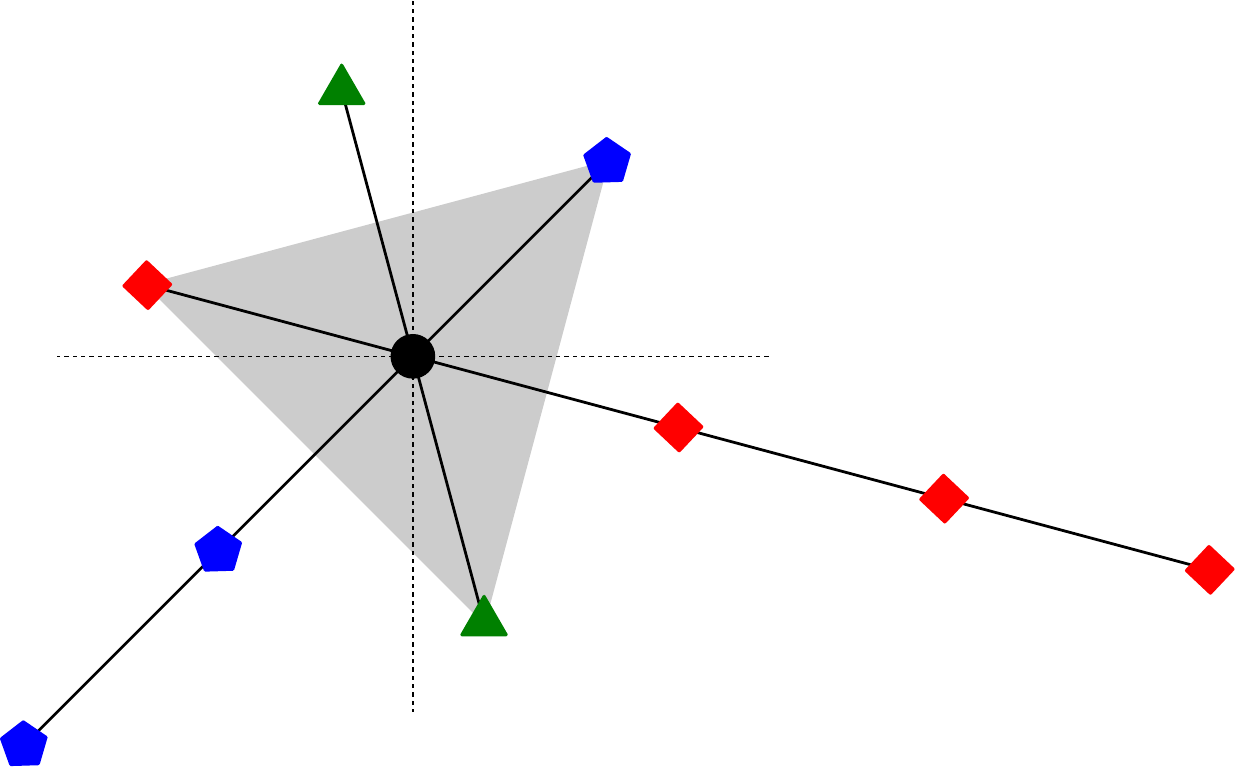}
\caption{Configuration from Example~\ref{ex:bound} for $\n=(2,3,4)$.}
\end{figure}

\begin{example}\label{ex:bound}
    Let $v_0,\dots,v_d \in \R^d$ be the vertices of a simplex containing the
    origin in its interior and for $n_0,\dots,n_d \ge 2$ define
    \[
        C_i \ := \ \{ v_i, -v_i, -2v_i,\dots,-(n_i-1)v_i \}
    \]
    for $0 \le i \le d$. Then $C = \{ C_0,\dots, C_d \}$ is a centered
    colorful configuration. The vertices of any colorful $d$-simplex $S$ of
    $C$ are minimally affinely dependent and hence the coefficients of this
    dependence are unique up to scaling. It follows that the coefficients are
    all of the same sign if and only if $S = \{v_0,\dots,v_d\}$ or $S =
    \{-\alpha_0v_0,\dots,-\alpha_d v_d \}$ for some $1 \le \alpha_i < n_i$.
    This yields exactly $1 + (n_0 - 1) \cdots (n_d-1)$ hitting simplices.
\end{example}

We will prove Theorem~\ref{thm:CC_upper} in Section~\ref{sec:upper} by means
of combinatorial topology. More precisely, we show that for centered colorful
configurations in relative general position, the number of hitting simplices
is related to the reduced Betti numbers of an associated simplicial complex,
the \emph{avoiding complex}. In particular, Theorem~\ref{thm:CC_upper} follows
from a result that shows that a certain Betti number is independent of the
configuration. To show this, we introduce the notion of \emph{flips} between
colorful configurations, and we prove our result by `flipping' any
configuration to the configuration given in Example~\ref{ex:bound}.

\subsection{Colorful Gale transforms of Minkowski sums}
The second goal of this paper is to highlight a connection between colorful
configurations and faces of Minkowski sums. The  \Defn{Minkowski sum}
of convex polytopes $P_0,\dots,P_\s \subset \R^d$ is the polytope
\[
    P \ = \ P_0 + \cdots + P_\s \ = \ \{ p_0 + \cdots + p_\s : p_i \in P_i
    \text{ for } 0 \le i \le \s \}
\]
and this operation is key to many deep results in many areas, notably convex
geometry~\cite{SchneiderNew} and computational commutative algebra
(eg.~\cite{GS}). The combinatorial complexity of Minkowski sums of polytopes
has been subject of several studies~\cite{AS2016,MatschkePfeiflePilaud2011,RS12}.

Using Gale transforms and Cayley embeddings, we define in
Section~\ref{sec:cayley} \emph{colorful Gale transforms} associated to a
collection $P_0,\dots,P_\s$ that, similar to ordinary Gale transforms, capture
the facial structure of Minkowski sums in the combinatorics of colorful
configurations. 

A particularly interesting case is that of Minkowski sums of simplices. A face
$F$ of the Minkowski sum $P$ is of the form $F = F_0 + \cdots + F_\s$ for
faces $F_i \subseteq P_i$. A face is \Defn{mixed} if $\dim F_i > 0$ for all
$i$ and \Defn{totally mixed} if each $F_i \subset P_i$ is an inclusion-maximal
face, that is, $F_i$ is a facet of $P_i$.  If each $P_i$ is a simplex, then we
show in Lemma~\ref{lem:char_tmf} that totally mixed facets are exactly the
hitting simplices containing $0$ in their interior of the associated colorful
Gale transform.  In particular, this allows us to prove the following. 

\begin{thm}\label{thm:simplices}
    For $n_0,\dots, n_\s \ge 1$, set $n = n_0 + \cdots + n_\s$ and let
    $P_i\subset \R^{n - \s - 1}$ for $0\leq i\leq \s$ be $n_i$-dimensional
    simplices whose Minkowski sum is full-dimensional.  Then the number of
    totally mixed facets of $P_0+\cdots+P_\s$ is at most
    \[
        1 + n_0 n_1\cdots n_\s.
    \]
\end{thm}

\subsection{A motivation from normal surface theory}
\newcommand\fan{\mathcal{F}}%
This dictionary between Minkowski sums and colorful configurations allows us
to resolve a conjecture of Ben Burton~\cite{Burton2003} about the complexity
of projective edge weight solution spaces in \emph{normal surface} theory. We
only describe the geometric formulation of the problem and refer
to~\cite[Ch.~5]{Burton2003} for the topological connections and implications.
Using the terminology of~\cite{Burton2003}, a \Defn{$d$-fan} $\fan$ is the
collection of three polyhedral cones $L_1, L_2, L_3 \subset \R^d$ such that
$L_i$ is linearly isomorphic to $\R^{d-2} \times \R_{\ge 0}$, i.e, each $L_i$
is a \emph{halfplane}, and such that   $L_i \cap L_j = L_1 \cap L_2 \cap L_3
\cong \R^{d-2}$ for all $i \neq j$. This is called the \Defn{axis} of the fan,
the cones $L_i$ are called \Defn{leafs}.  The fan is balanced if it partitions
$\R^d$ into three convex components and no two halfplanes are coplanar. In the
following all our fans will be balanced.  Given fans $\fan^0,\dots,\fan^\s$
with leafs $L^r_j$, their intersection is the collection of cones
\[
    L^0_{i_0} \cap \cdots \cap L^\s_{i_\s},
\]
for $i_0,\dots,i_\s = 1,2,3$. Define $U(d,k)$ to be the maximum of the number
of inclusion-maximal cones for the intersection of $k$ fans in general
position in $\R^d$. Burton conjectured the following.
\begin{conj}[{\cite[Conj.~5.5.14]{Burton2003}}]\label{conj:burton}
    For $d \ge 2$, $U(d,d-1) = 1 + 2^{d-1}$.
\end{conj}

To see why Theorem~\ref{thm:simplices} settles Conjecture~\ref{conj:burton},
let $P = \conv(u_1,u_2,u_3) \subset \R^d$ be a triangle. The \Defn{normal fan}
of $P$ is the collection of cones $N_i = \{ \ell \in (\R^d)^\ast : \ell(u_i)
\ge \ell(u_j) \text{ for all } j \}$ for $i = 1,2,3$. These are polyhedral
cones given as the intersection of exactly two distinct linear halfspaces and
and hence the intersections $N_i \cap N_j$ for $i \neq j$ are halfplanes. 
Thus, every triangle furnishes a fan in the sense of Burton and it is easy to
see that every balanced fan arises this way. Thus every collection of fans in
relative general position corresponds to a collection of triangles
$P_0,P_1,\dots,P_\s \subset \R^d$ in relative general position.  The totally
mixed faces of the Minkowski sum $P_0 + \cdots + P_\s$ are exactly the faces
that arise as the Minkowski sum of an edge from each $P_r$. Every linear
function maximizing on such a totally mixed face is thus in the intersection
of one leave from each $P_i$. 

\begin{cor}
    The number of inclusion-maximal cones in the intersection of $d-1$ fans in
    $\R^d$ is at most $1 + 2^{d-1}$.
\end{cor}

\subsection{Minkowski transforms}
We close in Section~\ref{sec:Ptransform} with an alternative construction of
colorful configurations associated to Minkowski sums via projections of		
polytopes. These are constructed using McMullen's
transforms~\cite{McMullen1979}. In contrast to the colorful Gale transforms,
these \emph{Minkowski transforms} are realized in lower dimensional spaces but
at the expense of not containing information about Minkowski subsums. In
particular, their construction uses inequality descriptions of polytopes and
not vertex sets, as in the case of colorful Gale transforms.

\textbf{Acknowledgements.}
The topological part of this paper started in Berlin in 2014 and was
independently completed in Berlin/Paris and Jerusalem. This paper is an
amalgamation of our efforts.  K.~Adiprasito was supported by NSF Grant DMS
1128155 and ERCAG 320924.  P.~Brinkmann was supported by DFG within the
research training group ``Methods for Discrete Structures'' (GRK1408).
A.~Padrol thanks the support of the INSMI (CNRS) through the program PEPS
Jeunes Chercheur-e-s 2016.  P. Pat\'{a}k's research was supported by the ERC
Advanced grant no.~320924.  Z. Pat\'akov\'a's work was supported by the Israel
Science Foundation grant ISF-768/12.  R.~Sanyal was supported by the DFG
Collaborative Research Center SFB/TR 109 ``Discretization in Geometry and
Dynamics''.

\section{An upper bound for colorful simplicial depth}\label{sec:upper}

In this section we will cast the problem of finding tight upper bounds on the
colorful simplicial depth into one of combinatorial topology. More precisely,
we will relate the bound of Theorem~\ref{thm:CC_upper} to the Betti numbers of
certain simplicial complexes (Proposition~\ref{lem:ub_betti}).  To set the
stage, fix some $\n=(n_0,\dots, n_d)\in \Z^{d+1}$ with $n_i\geq 2$ and
define the \Defn{$\n$-colorful complex} as the simplicial complex 
\[
    \Rain \ = \ \Rain(\n) \ := \ \D_0 \join \cdots \join \D_d,
\]
where $\D_i$ is a $0$-dimensional complex on $n_i$ vertices and $\join$ denotes
the join. 
We identify $\D_i$ with its set of vertices and we assume that $\D_i \cap \D_j
= \emptyset$. This way, the vertex set of $\Rain$ is $\D : = \D_0 \cup \cdots
\cup \D_d$. The simplices $\sigma \in \Rain$, called \Defn{colorful
simplices}, are exactly the subsets $\sigma \subset \D$ with $|\sigma \cap
\D_i | \le 1$ for all $0 \le i \le d$.

In this language, a \Defn{colorful $\n$-configuration}
corresponds to a map $\f : \D \rightarrow \R^d$, which induces a map $\f
: \Rain \rightarrow \R^d$ that linearly interpolates on each simplex $\sigma
\in \Rain$ and thus
\[
    \f(\sigma) \ := \ \conv( \f(v) : v \in \sigma).
\]
Conversely, any map $\f : \Rain \rightarrow \R^d$ that is linear on each
simplex gives a colorful configuration by setting $C_i := \{ \f(v) : v \in
\D_i \}$ and we will identify colorful configurations and affine maps from
$\Rain$ to $\R^d$.  Thus $\f$ is a \Defn{centered} colorful configuration if
$0 \in \relint\conv \f(\D_i)$ for $0 \le i \le d$, and $\f$ is in
\Defn{relative general position} if $0\notin\f(\tau)$ for any simplex $\tau
\in \Rain$ of dimension $< d$.  

The \Defn{avoiding complex} associated to $\f$ is the simplicial subcomplex 
\[
    \Av \ = \ \Avf{\f} \ := \ \{ \sigma \in \Rain : 0 \not\in \f(\sigma) \}.
\]
\newcommand\Hit{\mathcal{H}}%
We define the \Defn{hitting set} $\Hit = \Hit(\f)$ of $\f$ to be the set of
inclusion-minimal faces of $\Rain\setminus\Av$.  If $\f$ is in general
position, then $\Hit = \Rain \setminus \Av$ and $\Hit = \{
\eta_1,\dots,\eta_m\}$ is a collection of $d$-dimensional simplices of
$\Rain$, the \Defn{hitting $d$-simplices} of $\f$. In particular, the colorful
simplicial depth of a configuration $\f$ in relative general position is
$\csd(\f) = |\Hit(\f)|$.

\newcommand\rEuler{\tilde\chi}%
\newcommand\Chain[2]{\mathrm{C}_{#1}(#2)}%
\newcommand\Hom[2]{\tilde{\mathrm{H}}_{#1}(#2)}%
For a simplicial complex $\cS$, we write $\Chain{*}{\cS}$ for the
\Defn{chain complex} of $\cS$ with coefficients in $\Z_2$. That is, for $-1
\le k \le \dim \cS$, $\Chain{k}{\cS}$ is the $\Z_2$-vector space with
basis given by the $k$-dimensional faces $\sigma$ of $\cS$. We will freely
abuse notation and treat $k$-chains of $\cS$ as collections of
$k$-dimensional faces.  The boundary map is given by 
\[
    \partial \sigma \ = \ \sum \tau,
\]
where the sum is over all faces $\tau \subset \sigma$ with $\dim \tau = \dim
\sigma -1$.  The $k$-th homology group is denoted by $\Hom{k}{\cS}=\ker
\partial_k/\im\partial_{k+1}$ and we write $\betti_k(\cS) = \dim_{\Z_2}
\Hom{k}{\cS}$ for the $k$-th \Defn{reduced Betti number} of $\cS$. We will
write $\Sphere^{d}$ for the sphere of dimension $d$.

\begin{prop}\label{prop:Rain_betti}
    Let $\n = (n_0,\dots,n_d)$ with $n_i \ge 2$ for all $0 \le i \le d$. Then
    \[
        \betti_k(\Rain(\n)) \ = \ 0 \quad \text{ for $k < d$ } 
        \qquad \text{ and } \qquad 
        \betti_d(\Rain(\n) ) \ = \ (n_0-1)(n_1-1) \cdots (n_d-1).
    \]
\end{prop}
\begin{proof}
    For $0 \le i \le d$, the complex $\D_i$ is a wedge of $n_i-1$ spheres of
    dimension $0$. The complex $\Rain$ is by construction a join of
    $\D_0,\dots,\D_d$ and, since the join distributes over wedge
    (see~\cite[Ex.~6.2.1]{Using}), we have 
    \[
        \Rain \ \simeq \ \bigvee_{(n_0-1) \cdots (n_d-1)} \Sphere^d.
    \]
    This proves the claim.
\end{proof}

The main connection between the colorful simplicial depth of $\f$ and the
topology of $\Av(\f)$ is the following lemma.

\begin{lem}\label{lem:ub_betti}
    For $\n = (n_0,\dots,n_d)$ with $n_i \ge 2$, let $\f$ be a centered
    colorful $\n$-configuration in relative general position. Then 
    \[
        \csd(\f) \ \le \ \prod_{i=0}^d (n_i-1) + \betti_{d-1}(\Av).
    \]
\end{lem}
\begin{proof}
    Let $\Av = \Av(\f)$ be the avoiding complex of $\f$.  Using the
    Euler-Poincar\'e formula, we compute the reduced Euler characteristic
    $\rEuler(\Av)$ in two different ways.  The avoiding complex $\Av$ and the
    colorful complex $\Rain = \Rain(\n)$ agree on all faces up to dimension
    $d-1$ and $\Av$ is only missing the hitting simplices $\Hit = \Hit(\f)$.
    Thus, by Proposition~\ref{prop:Rain_betti},
    \[
        \rEuler(\Av) \  = \ \rEuler(\Rain) + (-1)^{d-1} |\Hit| \ = \
        (-1)^d \prod_{i=0}^d (n_i-1) + (-1)^{d-1} \csd(\f).
    \]
    On the other hand, $\Hom{k}{\Av} = \Hom{k}{\Rain}$ for $k < d-1$ and from
    Proposition~\ref{prop:Rain_betti} we get 
    \[
        \rEuler(\Av) \ = \ (-1)^{d-1}\betti_{d-1}(\Av) + (-1)^d
        \betti_{d}(\Av).
    \]
    This yields
    \[
        \csd(\f) \ = \ \prod_{i=0}^d (n_i-1) + \betti_{d-1}(\Av) -
        \betti_{d}(\Av). \qedhere
    \]
\end{proof}

By way of Lemma~\ref{lem:ub_betti}, Theorem~\ref{thm:CC_upper} follows
from the main result of this section.

\begin{thm}\label{thm:betti}
    Let $\n = (n_0,\dots,n_d)$ with $n_i \ge 2$. Then $\betti_{d-1}(\Avf{\f})
    = 1$ for any centered colorful $\n$-configuration $\f$ in relative general
    position.
\end{thm}

We note that $\betti_{d-1}(\Av(\f))$ is never trivial for centered
configurations in relative general position because $\Av(\f)$ retracts onto
the boundary of any hitting simplex.

\begin{prop}\label{prop:betti_pos}
    Let $\f$ be a centered colorful configuration in relative general
    position. Then
    \[
        \betti_{d-1}(\Avf{\f}) \ \ge \ 1.
    \]
\end{prop}
\begin{proof}
    By construction, the avoiding complex $\Av = \Avf{\f}$ satisfies $\f(\Av)
    \subset \R^d \setminus \{0\}$. Hence, if $\f$ is followed by the radial
    projection of $\R^d \setminus \{0\}$ onto the $(d-1)$-sphere
    $\Sphere^{d-1}$, we obtain a continuous map $\hat \f : \Av \rightarrow
    \Sphere^{d-1}$.  By the colorful Carath\'eodory Theorem~\ref{thm:CC},
    there is at least one colorful $d$-simplex $\eta \in \Hit = \Rain
    \setminus \Av$. The assumption of relative general position yields $0 \in
    \Int \f(\eta)$. Let $\Sigma$ be the boundary complex of $\eta$. Since the
    skeleta of $\Rain$ and $\Av$ coincide in all dimensions $\le d-1$, we have
    that $\Sigma \subseteq \Av$ is a subcomplex and $\hat \f$ restricted to
    $\Sigma$ is a homeomorphism. Thus,
    \[  
        \xymatrix{
            \Hom{d-1}{\Av} \ar@{<-}[d] & \Hom{d-1}{\Sphere^{d-1}}
            \ar@{<-}[l]_{\hat \f_*} \ar@{<-}[ld]_{id}\\ \Hom{d-1}{\Sigma}}
   \]
   commutes and, in particular, $\hat \f$ is surjective in homology. This
   shows that $\betti_{d-1}(\Av) >0$. 
\end{proof}

The proof of Proposition~\ref{prop:betti_pos} actually shows that the boundary
of any hitting simplex yields a non-trivial class in $\Hom{d-1}{\Av(\f)}$. The
following proposition shows that the hitting simplices even generate.

\begin{prop}\label{prop:Hit_gen}
    Let $\f$ be a centered colorful configuration in relative general position
    with avoiding complex $\Av = \Avf{\f}$ and hitting set $\Hit = \Hit(\f)$.
    Then $\Hom{d-1}{\Av}$ is generated by the cycles $\partial \eta$ for $\eta
    \in \Hit$.
\end{prop}
\begin{proof}
    Since $\Hom{d-1}{\Rain} = 0$, we have, by definition, $\im \partial_{d} =
    \ker \partial_{d-1}$. Since $\Chain{k}{\Rain} = \Chain{k}{\Av}$ for $k \le
    d-1$ and 
    \[
        \Chain{d}{\Rain} \ = \ \Chain{d}{\Av} \oplus \Z_2 \Hit
    \]
    the result follows.
\end{proof}

As a consequence of Proposition~\ref{prop:Hit_gen} and
Proposition~\ref{prop:betti_pos}, we note the following useful reformulation
of Theorem~\ref{thm:betti}.

\begin{cor}\label{cor:two_hitting}
    Let $\f$ be a centered colorful configuration in relative general position
    with avoiding complex $\Av = \Avf{\f}$.  Then $\betti_{d-1}(\Av)=1$ if and
    only if any two hitting simplices $\eta_1,\eta_2 \in \Hit(\f)$ are
    homologous. That is, there is some $B \in \Chain{d}{\Av}$ such that 
    \begin{equation}\label{eqn:Hit_hom}
        \partial \eta_1 + \partial \eta_2 \ = \  \partial B.
    \end{equation}
\end{cor}

\newcommand\flip{\stackrel{\rho}{\rightsquigarrow}}%
Let $\f, \f' : \Rain \rightarrow \R^d$ be two centered colorful configurations
in relative general position. We say that $\f'$ is obtained from $\f$ by a
\Defn{flip} if there is a continuous family $(\f_t : \Rain \to \R^d)_{-1 \le t
\le 1}$ of centered configurations such that $\f_{-1} = \f$ and $\f_{1} = \f'$
and $\f_t$ is in relative general position at all times except for $t = 0$. At
time $t=0$, the unique minimal face $\rho \in \Rain$ with $0 \in \relint
\f_0(\rho)$ is of dimension $d-1$. To emphasize $\rho$, we shall say that $\f$
and $\f'$ are related by a flip over $\rho$ and we write $\f \flip \f'$.
Informally, think of a flip as a continuous motion of $\f$ to $\f'$ such that
the origin passes through the ridge $\rho$.

\begin{figure}[htpb]
    \includegraphics[width=.4\linewidth]{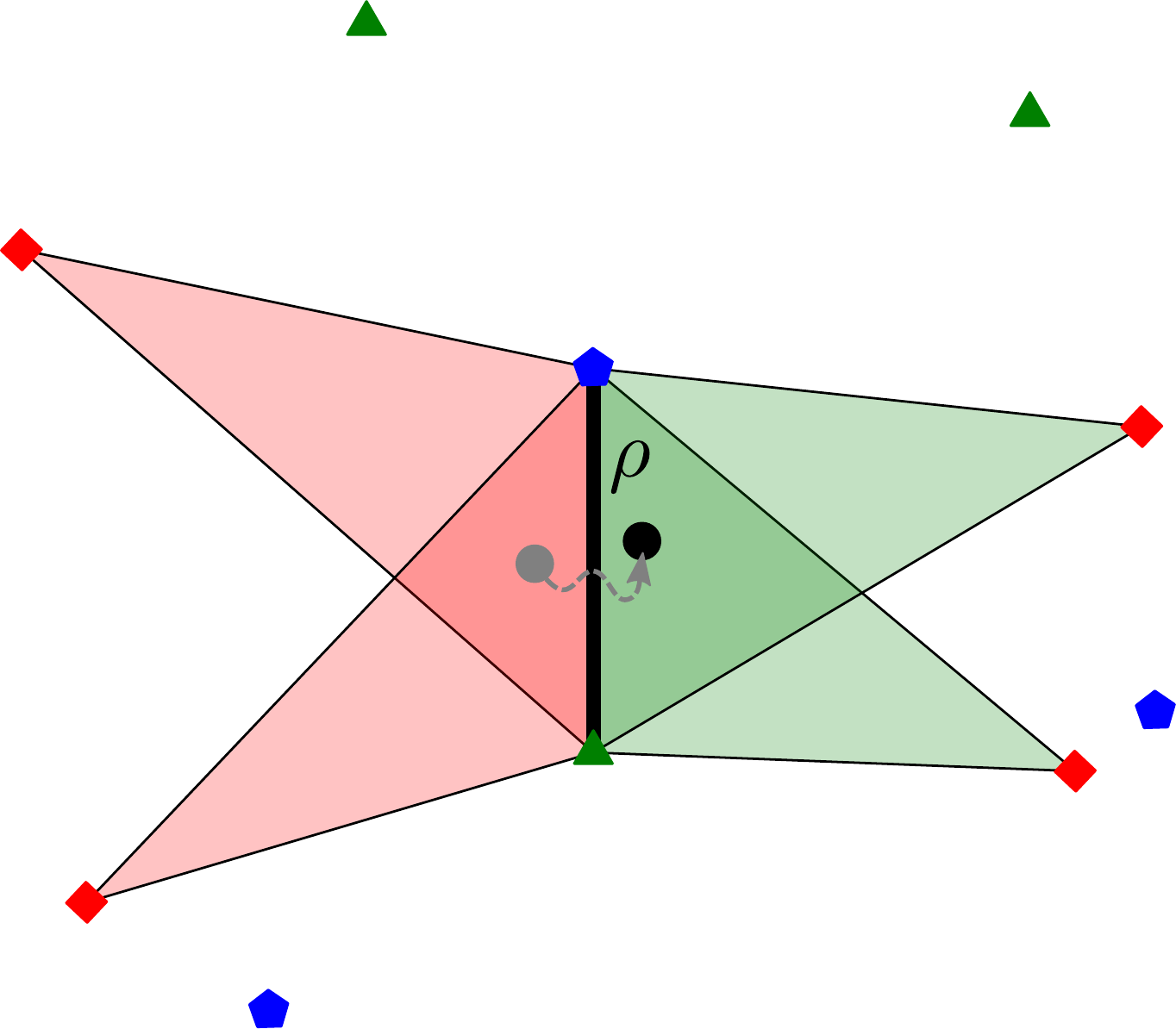}
    \caption{One can sometimes think of a $\rho$-flip as if the origin moved
    continuously through $\rho$ without crossing any other colorful ridge.
    Colorful simplices containing $\rho$ change their status, the remaining
    are not affected.}
\end{figure}

\begin{prop}\label{prop:combinatorial_flip}
     Let $\f \flip \f'$ be two configurations in relative general position
     related by a flip in~$\rho$ and let $\Av,\Av'$ and $\Hit,\Hit'$ be the
     respective avoiding complexes and hitting sets.  A colorful $d$-simplex
     $\sigma$ is in the symmetric difference $\Hit \Symm \Hit' = (\Hit
     \setminus \Hit') \cup (\Hit' \setminus \Hit)$ if and only if $\rho
     \subset \sigma$. Thus, a colorful $d$-simplex $\sigma \in \Av \cup \Av'$
     is contained in $\Av \cap \Av'$ if and only if $\rho$ is not a face of
     $\sigma$.
\end{prop}
\begin{proof}
     Let $(\f_t)_{-1 \le t \le 1}$ be the continuous family that realizes the
     flip over $\rho$. For $\eps > 0$ sufficiently small, every colorful
     $d$-simplex $\sigma$ that contains $\rho$ is hitting for $\f_{-\eps}$ if
     and only if it is not hitting for $\f_{+\eps}$.  Hence, the colorful
     simplices in the symmetric difference $\Hit \Symm \Hit'$ are exactly the
     colorful $d$-simplices that contain $\rho$.
\end{proof}

The benefit of working with flips is that we can pass from one colorful
configuration to another.

\begin{prop}\label{prop:connectedflipgraph}
    Any two centered colorful $\n$-configurations in relative general position
    are connected by a sequence of flips.
\end{prop}
\begin{proof}
    \newcommand\F{\mathbf{F}}%
    Let us denote by $\F \subset \{ \f : \D \rightarrow \R^d \}$ the space of
    all centered colorful configurations. This is an open semialgebraic set of
    full dimension $d (n_0 + \cdots + n_d)$.  The set $X \subset \R^{\D}$ of
    colorful configurations $\f$ such that $0 \in \f(\rho_1) \cap \f(\rho_2)$
    for two distinct colorful $(d-1)$-simplices $\rho_1, \rho_2 \in \Rain$ is
    of codimension $2$.  Indeed, if we forget about containment and only
    require that $0 \in \aff \f(\rho_1) \cap \aff \f(\rho_2)$, then the
    dimension can only increase.  For two fixed $\rho_1$ and $\rho_2$, pick
    $a_1 \in \rho_1 \setminus \rho_2$ and $a_2 \in \rho_2 \setminus \rho_1$.
    We can choose $\f(p) \in \R^d$ freely for any $p \in \D \setminus \{
    a_1,a_2\}$ and we have to choose $\f(a_i)$ in the linear span of $\{ f(p)
    : p \in \rho_i, p \neq a_i \}$, which is a linear space of dimension at
    most of dimension $d-2$. This yields $d (n_0 + \cdots + n_d) - 2$ degrees
    of freedom and $X$ is contained in a finite union of these sets as
    $\rho_1$ and $\rho_2$ vary over all distinct $(d-1)$-simplices of $\Rain$.
    It follows that  $\F \setminus X$ is path connected, which proves the
    claim.
\end{proof}

The crucial insight now is that $\betti_{d-1}(\Av(\f))$ is invariant as $\f$
undergoes a flip. 

\begin{lem}\label{lem:flip}
    Let $\f \flip \f'$ be two centered colorful configurations in general
    position that are related by a flip in $\rho$. If $\beta_{d-1}(\Avf{\f}) =
    1$, then $\beta_{d-1}(\Avf{\f'}) = 1$.
\end{lem}
\begin{proof}
    Let $\Av, \Av'$ be the avoiding complexes of $\f$ and $\f'$, respectively,
    and write $\Hit$ and $\Hit'$ for their hitting sets. In light of
    Corollary~\ref{cor:two_hitting}, we want to show that for any $\eta_1,
    \eta_2 \in \Hit'$ there is a $d$-chain $B' \in \Chain{d}{\Av'}$ such that
    $\partial \eta_1  + \partial \eta_2  = \partial B'$.  For the proof, let
    us call a hitting simplex $\eta$ \emph{old} if $\eta \in \Hit \cap \Hit'$
    and \emph{new} if $\eta \in \Hit' \setminus \Hit$. By
    Proposition~\ref{prop:combinatorial_flip}, a hitting simplex $\eta$ is new
    if and only if $\rho \subseteq \eta$.  Without loss of generality, $\rho =
    \{ a_1, a_2, \dots, a_d\}$ with $a_i \in \D_i$ for $1 \le i \le d$. That
    is, $\rho$ is a colorful ridge that avoids $\D_0$. Our proof distinguishes
    three cases, depending on whether $\eta_1$ and $\eta_2$ are new or old. It
    follows from the (polyhedral) Morse Lemma that only new hitting simplices
    can create homology and, strictly speaking, we only need to consider the
    `new/new' case. However, to have a self-contained proof, we include all
    cases.

    \emph{New/New:} Assume that $\eta_1,\eta_2\in \Hit'\setminus\Hit$. Then
    there are $a_0,b_0 \in \D_0$ such that $\eta_1 = a_0 \join \rho = \{
    a_0,a_1,\dots, a_d\}$ and $\eta_2 = b_0 \join \rho$. Let us consider the
    critical configuration $\f_t$ at time $t=0$.  Let $H$ be the linear
    hyperplane spanned by $\f_0(\rho)$ and denote by $H^+$ the closed halfspace
    containing $\f_0(a_0)$ and $\f_0(b_0)$. The points must lie in the same
    halfspace, because $\eta_1$ and $\eta_2$ are hitting in $\f_\eps$ for
    $\eps > 0$.  Since $\f$ is centered, there are $b_i \in \D_i$ such that
    $\f_0(b_i) \in H^+$ for all $1 \le i \le d$. Define the subcomplex 
    \[
        \Oct(\eta_1,\eta_2) \ := \ \{a_0,b_0\} \join \{a_1,b_1\} \join \cdots
        \join \{a_d,b_d\} \ \subset \ \Rain,
    \]
    which, as an abstract simplicial complex, is the boundary of a
    $(d+1)$-dimensional crosspolytope.  Except for $\eta_1$ and $\eta_2$, every
    simplex $\sigma\in \Oct(\eta_1,\eta_2)$ belongs to $\Av \cap \Av'$.
    Indeed, they are avoiding already in $\f_0$ by construction, because they
    lie all in the same linear halfspace and only $\rho$ contains $0$ in its
    relative interior. Moreover, every avoiding simplex of $\f_0$ is also an
    avoiding simplex of $\f_t$ for $-1\leq t\leq 1$, since only the simplices
    that contain $\rho$ as a face change their status in the course of the
    flip.  We call the $d$-chain
    \[
      \Zil(\eta_1,\eta_2) \ := \ \sum\{ \sigma : \sigma \in \Oct(\eta_1,\eta_2)
      \setminus \{ \eta_1, \eta_2 \}, \dim(\sigma)=d \} \ \in \
      \Chain{d}{\Av'}
    \] 
    a \Defn{cylinder} of $\eta_1$ and $\eta_2$ and we observe that 
    \[
        0 \ = \ \partial \Oct(\eta_1,\eta_2) \ = \ \partial
        \Zil(\eta_1,\eta_2) + \partial \eta_1 + \partial \eta_2,
    \]
    which shows that $\partial \eta_1 + \partial \eta_2$ is a boundary and
    finishes the first case. \hfill $\diamond$

    \emph{Old/Old:} Assume that  $\eta_1,\eta_2\in \Hit'\cap\Hit$ and thus
    $\rho$ is not a face of neither $\eta_1$ nor $\eta_2$. Since $\eta_1$ and
    $\eta_2$ are hitting for $\f$, the assumption and
    Corollary~\ref{cor:two_hitting} assure us that there is a $d$-chain $B \in
    \Chain{d}{\Av}$ such that 
    \begin{equation}\label{eq:old_chain}
        \partial \eta_1 + \partial \eta_2 \ = \ \partial B \ = \ \partial
        \left(\sum_{\sigma\in B\cap \Av'}\sigma + \sum_{\sigma\in B\cap
        \Hit'}\sigma\right).
    \end{equation}
    We claim that the number of simplices in $B\cap \Hit'$ is even. To this
    end, consider the cochain $\rho^\ast : \Chain{d-1}{\Rain}\to \Z_2$ with
    $\rho^\ast(\rho) = 1$ and $=0$ otherwise. The coboundary $\delta \rho^\ast
    = \rho^\ast \circ \partial$ satisfies that for a colorful $d$-simplex $\sigma
    \in \Rain$
    \[
        \delta \rho^\ast (\sigma) \ = \ 1 \quad \text{ if } \rho \subset
        \sigma \qquad \text{ and } \qquad \delta \rho^\ast (\sigma) \ = \ 0
        \quad \text{ otherwise. }
    \]
    Since neither $\eta_1$ nor $\eta_2$ contain $\rho$, we compute
    from~\eqref{eq:old_chain}
    \[
        0 \ = \ \rho^\ast(\partial \eta_1 + \partial \eta_2)  \ = \ 
        \delta \rho^\ast(B) \ = \ 
        \sum_{\sigma\in B\cap \Av'}\delta \rho^\ast(\sigma) +\sum_{\sigma\in
     B\cap \Hit'}\delta \rho^\ast(\sigma).
    \]
    Because $\delta \rho^\ast$ evaluates to $0$ on the simplices of $B\cap
    \Av'\subseteq\Av\cap \Av'$ and to $1$ on the simplices of
    $B\cap\Hit'\subset\Hit'\setminus \Hit$
    (cf.~Proposition~\ref{prop:combinatorial_flip}), this shows that $|B\cap
    \Hit'|$ is even. Let $B\cap \Hit' = \{ \sigma_1,\dots,\sigma_{2k} \}$.
    Since each $\sigma_i$ is a \emph{new} hitting simplex, we can use the
    previous case to find cylinders $\Zil(\sigma_{2i-1},\sigma_{2i}) \in
    \Chain{d}{\Av'}$ such that 
    \[
        \partial \sum_{\sigma\in B\cap \Hit'} \sigma \ = \ \sum_{i=1}^k
        \partial(\sigma_{2i-1} + \sigma_{2i}) \ = \ \sum_{i=1}^k \partial
        \Zil(\sigma_{2i-1},\sigma_{2i}),
    \]
    which together with~\eqref{eq:old_chain} shows that $\partial \eta_1 +
    \partial \eta_2$ is a boundary in $\Av'$. \hfill $\diamond$

    \emph{Old/New:}  Assume that $\eta_1\in \Hit'\cap\Hit$ and $\eta_2\in
    \Hit'\setminus\Hit$ and thus $\rho \nsubseteq \eta_1$ but $\rho \subseteq
    \eta_2$. Again, let $a \in \D_0$ such that $\eta_2 = a \join \rho$.  Let
    $H$ be the linear hyperplane parallel to $\aff(\f'(\rho))$.  Since $\f'$
    is in relative general position, $\f'(\eta_2)$ is a $d$-simplex and $H$
    separates $\f'(a)$ from $\f'(\rho)$. By the centeredness, there must be
    some $b \in \D_0$ such that $H$ separates $\f'(a)$ and $\f'(b)$. In
    particular, the colorful simplex $\eta_2' := b \join \rho$ is contained in
    $\Hit\setminus\Hit'$; that is, $\eta_2'$ an element of $\Av'$
    containing~$\rho$. 

    Hence, $\eta_1, \eta_2'$ is a pair of hitting simplices for $\f$ and,
    by the assumption and Corollary~\ref{cor:two_hitting}, there is a chain $B
    \in \Chain{d}{\Av}$ such that $\partial \eta_1 + \partial \eta_2'  =
    \partial B$.  Define $B' := B + \eta_2 + \eta_2'$. This is a chain
    in $\Chain{d}{\Rain}$ and
    \[
        \partial B' \ = \ \partial B + \partial \eta_2 + \partial \eta_2' \ =
        \ \partial \eta_1 + \partial \eta_2.
    \]
    Using the cochain $\rho^\ast$, we observe that $\rho^\ast( \partial \eta_1
    + \partial \eta_2') = 1$ and hence $B \cap \Hit'$ contains an odd number
    of elements. Since $\eta_2\in\Hit'$ and $\eta_2'\notin\Hit'$, it follows that
    $B'=B+\eta_2+\eta_2'$ contains an even number of elements in $\Hit'$ and, by pairing them
    up with cylinders, there is some $B'' \in \Chain{d}{\Av'}$ such that $\partial B' =
    \partial B''$. 
\end{proof}

\begin{proof}[Proof of Theorem~\ref{thm:betti}]
    By Lemma~\ref{lem:flip} and Proposition~\ref{prop:connectedflipgraph}, it
    suffices to produce a single instance of a centered colorful configuration
    in relative general position for every $\n$ whose avoiding complex
    satisfies $\betti_{d-1}(\Av)=1$.

    Note that the centered configurations of Example~\ref{ex:bound} are in
    relative general position.  From the description of the hitting simplices
    in Example~\ref{ex:bound}, we can get their avoiding complexes. For $0 \le
    i \le d$, let $\D_i$ be the $0$-dimensional complex with vertices $a^i,
    b^i_1,\dots,b^i_{n_i-1}$.  Then $\sigma \in \Rain(\n) = \D_0 \join \cdots
    \join \D_n$ is an avoiding $d$-simplex if and only if it contains at least
    one $a^i$ and at least one $b^j_{l}$ for $i \neq j$ and $1 \le l < n_j$.
    We call the vertices $a^i$ \Defn{special} and we set $\Av_{d+1}$ to be the
    boundary complex of the hitting simplex $\{a^0,\dots,a^d\}$. We will show
    that $\Av$ is homotopy equivalent to $\Av_{d+1}$ and hence $\Hom{*}{\Av} =
    \Hom{*}{\Sphere^{d-1}}$.

    Recall that if $\sigma$ is an inclusion-maximal face of a simplicial
    complex $\Delta$, and $\tau \subset \sigma$ is a face with $\dim \tau <
    \dim \sigma$ that is not contained in any other inclusion-maximal face of
    $\Delta$, then $\tau$ is called a \Defn{free face} of $\sigma$. If $\tau$
    is a free face of $\sigma$, the removal of all the simplices $\tau
    \subseteq \gamma \subseteq \sigma$ of $\Delta$ is called a
    \Defn{collapse}, and denoted by $\Delta \searrow \Delta'$, where $\Delta'
    = \Delta \setminus \{\gamma: \tau \subseteq \gamma \subseteq \sigma\}$. We
    will use the same notation to denote a sequence of collapses. If $\Delta
    \searrow \Delta'$ then $\Delta$ and $\Delta'$ are homotopy equivalent;
    cf.~\cite[Sect.~11.1]{TopMeth}. 

    For $1 \le k \le d$, let $\Av_k \subseteq \Av$ be the subcomplex generated
    by the $d$-simplices with at least $k$ special vertices.  We will show
    that 
    \begin{equation}\label{eq:collapses}
    \Av \ = \ \Av_1 \ \searrow \
    \Av_2 \ \searrow \ \cdots \ \searrow \ \Av_d \ \searrow \ \Av_{d+1}
    \ \cong \ \Sphere^{d-1}.
 \end{equation}

    For $1 \le k \le d-1$, let $\sigma \in \Av_k$ be a $d$-simplex containing
    exactly $k$ special vertices and let $\tau \subset \sigma$ be any ridge
    containing $k-1$ special vertices. Since every $\D_i$ contains exactly one
    special vertex and $\tau$ misses only one color class $\D_i$, $\tau$ is a
    free face of $\sigma$. Hence, it can be collapsed. Continuing this way
    shows $\Av_k \searrow \Av_{k+1}$. In the last step, the simplices $\sigma
    \in \Av_d$ contain a single non-special vertex, which is itself a free
    face of $\sigma$ and can be collapsed. What remains are the
    $(d-1)$-simplices spanned by special vertices, which is exactly
    $\Av_{d+1}$.
\end{proof}

 \begin{figure}[htpb]
 \includegraphics[width=.3\linewidth]{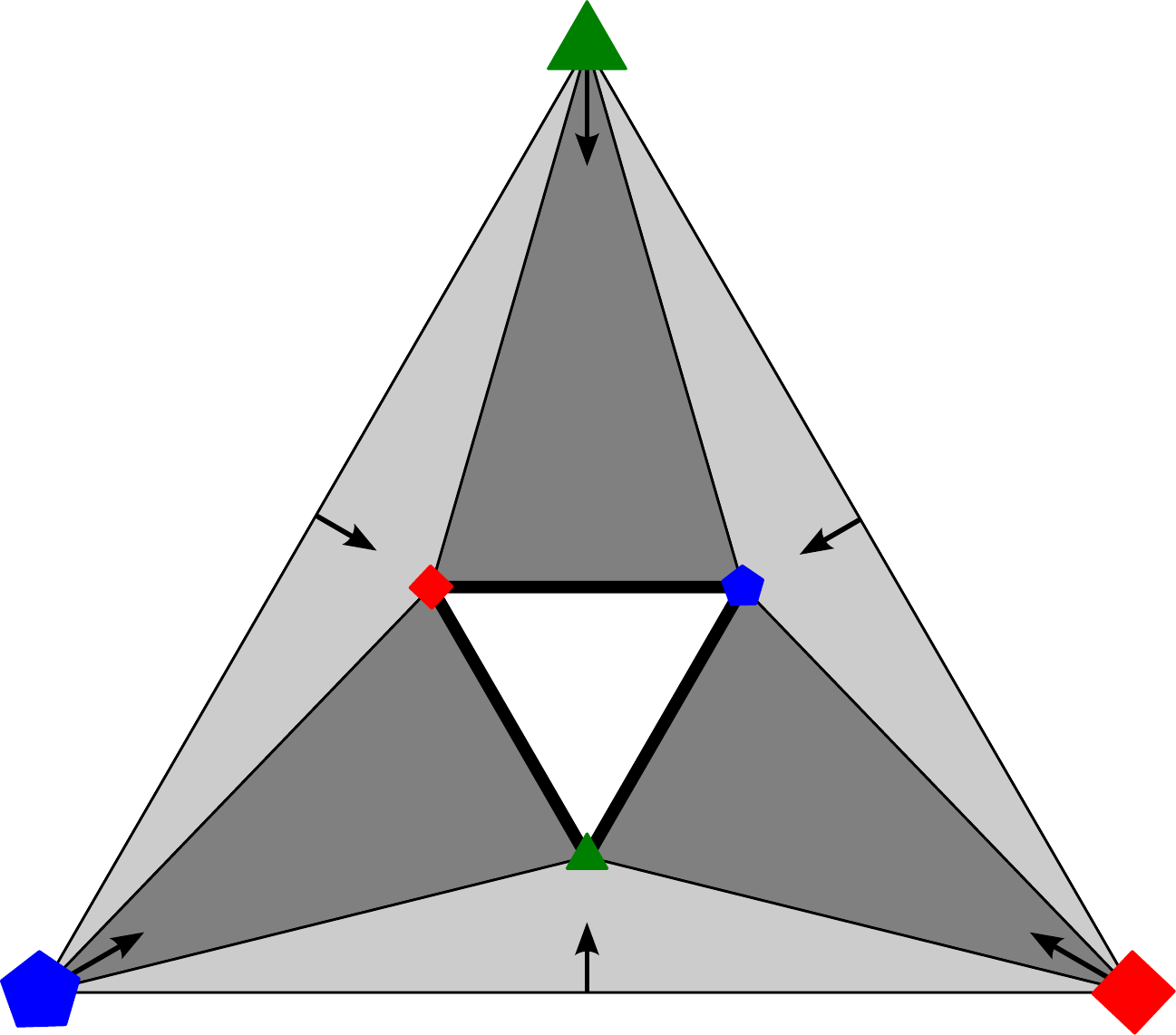}
 \caption{Scheme of $\Av$ for Example~\ref{ex:bound} for $d=2$. The large nodes 
 on the outer triangle should be thought of as representatives of the non-special 
 vertices of that color. Since no simplex of $\Av$ has two vertices of the same color, 
 this depicts the whole complex. Lightly shaded triangles belong to $\Av_1$ and darkly shaded triangles to $\Av_2$. 
 The boundary of the interior triangle, drawn thicker, is $\Av_3$. The collapses for the proof of Theorem~\ref{thm:betti} are also depicted.}
 \end{figure}

The Universal Coefficient Theorem together with Hurewicz' theorem
(cf.~\cite{Hatcher}) shows that all homotopy groups of $\Av(\f)$ up to
dimension $d-1$ agree with that of the $(d-1)$-sphere. 
\begin{Conj}
    For a centered colorful configuration $\f$ in relative general position,
    the avoiding complex $\Av(\f)$ $d$-collapses (in the sense of
    Wegner~\cite{Weg}) onto the boundary of a hitting simplex.
\end{Conj}

For the Colorful Carath\'eodory Theorem, the assumption of \emph{centeredness}
can be replaced by a weaker assumption (see, for
example,~\cite{ArochaBaranyBrachoFabilaMontejano2009, bar82,
HolmsenPachTverberg2008, MeunierDeza2013}). This prompts the following
question.
\begin{quest}\label{quest:weaker}
    Under which conditions other than centeredness is $\betti_{d-1}(\Av) = 1$?
\end{quest}

\section{From colorful configurations to Minkowski sums}
\label{sec:cayley}
\newcommand\I{\D}%
\newcommand\A{A}%
\newcommand\G{G}%
\newcommand\wh{\widehat}%
In this section we will show how to use Gale duality and the Cayley trick
to translate between colorful configurations and Minkowski sums. Although we
are mainly interested in the case of polytopes, our setup works with arbitrary
point configurations (not necessarily in convex position).  Let $\I$ be a
finite set. A \Defn{point configuration} in $\R^d$ indexed by $\I$ is a map
$\A : \I \rightarrow \R^d$. Note that our point configurations can have
multiple points. A subset $\U \subseteq \I$ is a \Defn{face} of $\A$ if there
is an affine function $\ell : \R^d \to \R$ such that $\ell(A(v)) \ge 0$ and
$>0$ if and only if $v \not\in \U$. That is, we identify faces of the convex
hull $\conv \A(\I)$ with the indices of points contained in them.

If we further assume that the affine span of $A(\I)$ is full-dimensional, then
$\A$ induces a surjective linear map $\wh A : \R^\I \rightarrow \R^{d+1}$ by
sending $e_v \mapsto (A(v),1)$ for $v \in \I$ and extending by linearity. The
map $\wh A$ has a (possibly trivial) kernel $\ker(\hat A) \cong \R^{|\I|-d-1}$
and let $\iota :\ker(\wh A) \hookrightarrow \R^{\I}$ denote the inclusion map.
The adjoint to $\iota$ is a linear map
\[
    \iota^\ast : (\R^{\I})^\ast \twoheadrightarrow (\ker(\wh A))^\ast.
\]
Identifying $(\ker(\wh A))^\ast$ with $\R^{|\I|-d-1}$, we get another
configuration $\G : \I \rightarrow \R^{|\I|-d-1}$, the \Defn{Gale
transform}\footnote{Actually, $\G$ is typically treated as a vector
configuration.} of $\A$ given by $G(v) = \iota^*(e_v^*)$ for $v \in I$.

\begin{lem}[Gale duality]\label{lem:gale}
    Let $\A : \I \rightarrow \R^d$ be a (full-dimensional) point configuration and $\G : \I
    \rightarrow \R^{|\I|-d-1}$ a Gale transform of $\A$. Then $\U \subseteq
    \I$ is a face of $\A$ if and only if
    \[
        0 \in \relint \conv( G(i) : i \in \I \setminus \U ).
    \]
\end{lem}

Gale transforms and Gale duality are powerful tools in discrete geometry and
we refer to~\cite[Section~5.6]{Matousek2002} for a very accessible treatment
in terms of matrices and to~\cite[Lecture~5]{ziegler} for its relation to
oriented matroids. In particular, the faces of the polytope $\conv \A(\I)$ are
completely described by the Gale transform $\G$.

Let us assume that $\I = \I_0 \cup \I_1 \cup \cdots \cup \I_\s$ for some $\s >
0$ and $\I_i \cap \I_j = \emptyset$ for $i \neq j$. Then, similar to
Section~\ref{sec:upper}, we consider the point configuration $\A$ as a family of $\s+1$ separate configurations $\A_i:\I_i\to\R^d$. Their \Defn{Minkowski sum}
is the point configuration:
\newcommand\Mink[2]{\sum \nolimits_{#1} {#2}}%
\newcommand\Mk{\mathrm{Mink}}%
\begin{align*}
 \Mink{\I}{\A}:\I_0\times\cdots\times \I_\s &\to  \R^d\\
 (v_0,\dots,v_\s)&\mapsto \A(v_0)+\cdots+\A(v_\s).
\end{align*}
We can restrict this configuration to any subset $\U\subset \I$ with $\U_i = \U \cap \I_i\neq\emptyset$. It is easy to check that every face
of $\sum\A$ is %
of the form $\Mink{\U}{\A}$ for some $\U
\subseteq \I$ such that $\U_i$ is a face of $\A_i$.

\newcommand\Ca[1]{\mathrm{Cay}({#1})}%
\newcommand\CaI[2]{\mathrm{Cay}_{{#1}}({#2})}%
We may adapt the notion of Gale transforms to Minkowski sums by way of Cayley
embeddings. The \Defn{Cayley embedding} of $\A : \I \rightarrow \R^d$ is the
configuration $\CaI{\I}{\A} : \I \rightarrow \R^\s \times \R^d$ that takes
$v_i \in \I_i$ to $(e_i,\A(v_i))$ for $0 \le i \le \s$. Here, $e_0 := 0$ and
$e_1,\dots,e_\s$ is the standard basis of $\R^{\s}$.  Alternatively, you can
also replace $e_0,\dots,e_\s$ with the vertices of any full-dimensional
simplex in $\R^\s$. 

Let $b = \frac{1}{s+1} (e_0 + \cdots + e_\s)$ be the barycenter of
$e_0,\dots,e_\s$ and consider the affine subspace $\Lambda = \{ (x,y) \in \R^s
\times \R^d : x = b \}$. Then it is straightforward to check that 
\[
    \conv( \CaI{\I}{\A})  \cap \Lambda \ \cong \ \conv\left(\Mink{\I} \A\right).
\]

In particular, this induces a bijection between faces of $\CaI{\I}{\A}$ and
$\Mink{\I} \A$.  Cayley embeddings have many favorable properties, in
particular in relation to triangulations and mixed subdivisions;
see~\cite{HRS2000}.

\begin{lem}
    Let $\U=\U_0\cup \cdots\cup \U_s \subseteq \I$ be such that $\U_i \neq
    \emptyset$ for all $0 \le i \le \s$.  Then $\U$ is a face of
    $\CaI{\I}{\A}$ if and only if $\U_0\times \cdots\times \U_s$ is a face of
    $\Mink{\I} \A$.
\end{lem}

Let $\G$ be a Gale transform of $\CaI{\I}{\A}$. By construction, $\G$ is a colorful
point configuration and we refer to it as a \Defn{colorful Gale transform}.
If $\A$ is a configuration of $N = |\I| = |\I_0| + \cdots + |\I_\s|$ points in
$\R^d$, then $\G$ is a colorful configuration of $N$ points in $(N - d - s -
1)$-dimensional space.  Since $\I \setminus \I_j$ is a face for all $0 \le j
\le \s$, it follows from Lemma~\ref{lem:gale} that $0 \in \relint\conv\G(\I_j)$ and hence
$\G$ is actually a centered configuration. This allows us to construct a
dictionary between the facial structure of $\Mink{\I} \A$ and properties of the
colorful configuration $\G$. We start with a characterization that shows that every centered colorful configuration can be thought as a colorful Gale transform.  

A configuration $\G : \I \rightarrow \R^d$ is \Defn{positively spanning} if
$\G(\I)$ is not contained in a closed linear halfspace; and it is
\Defn{positively $\mathbf{2}$-spanning} if for any $v \in \I$, the restriction
of $\G$ to $\I \setminus \{v\}$ is positively spanning.  A configuration $\G'
: \I \to \R^d$ is a \Defn{positive rescaling} of $\G$ if there is $\lambda :
\I \to \R_{>0}$ such that $\G'(v) = \lambda(v) \G(v)$ for all $v \in \I$. We
say that $\G$ and $\G'$ are \Defn{positively equivalent} if they are related
by a positive rescaling. Notice that in this case $\G'$ is centered,
positively spanning or positively $2$-spanning whenever $\G$ is. The Gale
transform $\G$ and $\G'$ of $\A$ and $\A'$ are positively equivalent if and
only if $\A$ and $\A'$ are projectively
isomorphic~\cite[5.4.vi]{Gruenbaum2003}.

The following characterization of colorful Gale transforms is implicit
in~\cite[Prop.~2.9]{NillPadrol2015}.

\begin{lem}\label{lem:Galedualcharacterisation}
    For $\I = \I_0 \cup \cdots \cup \I_\s$.  The Gale transform $\G : \I \to
    \R^{|\I| - d - s - 1}$ of a (full-dimensional) Cayley embedding $\A : \I
    \to \R^{s + d}$ is centered.  Conversely, every
    (full-dimensional) centered configuration is
    positively equivalent to the Gale transform of a Cayley embedding.

    Moreover,  the points $\A(\I_j)$ are in convex position for every $0 \le j
    \le \s$ if and only if $\G$ is positively $2$-spanning.
\end{lem}

As in the case of ordinary Gale duality, the facial structure of the Minkowski
sum of a colorful configuration can be read from the colorful Gale dual. The
following dictionary follows the same reasoning as for classical Gale duality
and we therefore omit the proof.

\begin{prop}[Colorful Gale duality]\label{prop:dict}
    Fix $\I = \I_0 \cup \cdots \cup \I_\s$. Let  $\A : \I \to \R^d$ be a
    configuration such that $\Mink\I\A$ is full-dimensional and
    let $\G : \I \to \R^{|\I| - d - s - 1}$ be a corresponding colorful Gale
    transform. For $\U \subseteq \I$ such that $\U_i \neq \emptyset$ for all
    $0 \le i \le d$ the following holds
    \[
    \Mink{\U}{\A} \text{ is a face of } \Mink{\I}{\A}
        \qquad \text{ if and only if } \qquad
        0 \in \relint \G( \I \setminus \U ).
    \]
\end{prop}

Note that the colorful Gale transform actually encodes also
the faces of all \emph{subsums} of $\A_0,\dots,\A_\s$, which are in bijection to
the faces of $ \CaI{\I}{\A}$.
Indeed, for $\U\subseteq \I$ let $I\subset\{0,\dots,\s\}$ be the set of indices 
such that $\U_i=\emptyset$ and set $\I_I=\bigcup_{i\in I}\I_i$ and $\U_I=\bigcup_{i\in I}\U_i$. Then $\U_I$ is a face of the Minkowski sum $\Mink{\I_I}{\A}$ if and only if $0 \in \relint \G( \I \setminus \U )$.

In the next section we will see another way to associate a colorful
configuration to a Minkowski sum. In contrast to the colorful Gale transform,
that colorful configuration is realized in lower dimensions but it will only
remember faces of the full Minkowski sums and, for example, will not contain
the face structure of the summands. 

The following result, which is obtained from the duality between deletions and 
contractions of Gale transforms~\cite[Section~6.3(d)]{ziegler}, gives a glimpse at this connection by showing how to 
recover the Gale transform of each of the summands from the colorful Gale transform.

\begin{prop}\label{prop:simplex_characterization}
    For $\I = \I_0 \cup \cdots \cup \I_\s$, let $\A : \I \rightarrow \R^d$ be
    a configuration and let $\G : \I \to \R^{|\I|-d-s-1}$ be an associated
    colorful Gale transform. For $0 \le i \le \s$, a Gale transform of $\A_i
    : \I_i \to \R^d$ is given by the projection $\G_i : \I_i \to
    \R^{|\I|-d-s-1} / L$, where $L$ is the linear span of $\G(\I \setminus
    \I_i)$.

    In particular, the points $\A_i$ are affinely independent if and only if
    $\G(\I \setminus \I_i)$ is spanning.
\end{prop}

Notice that this last property is always fulfilled by colorful configurations
in relative general position.  Indeed, if $\G(\I\setminus\I_i)$ was not full
dimensional, we could apply the Colorful Carath\'eodory Theorem to find a
colorful ridge containing the origin.

If for all $0 \le i \le \s$ the points of $\A_i$ are affinely independent,
then $P_i := \conv \A(\I_i)$ is a simplex and $\conv \Mink{}{\A} = P_1 +
\cdots + P_\s$ is a Minkowski sum of simplices. 

\begin{cor}\label{cor:dict_genericcentered}
    Let $\I = \I_0 \cup \cdots \cup \I_d$ with $n_i = |\I_i| \ge 2$ for all $0
    \le i \le d$ and set $n = n_0 + \cdots + n_d$.    Then every centered
    colorful configuration of $\G : \I \rightarrow \R^d$ in relative general
    position is the colorful Gale transform of a configuration $\A : \I \to
    \R^{n - 2d -1}$ such that $\Mink{}{\A}$ is full-dimensional and $\A_i$ is
    affinely independent for all $0 \le i \le d$.
\end{cor}
A converse statement also holds, for the adequate notion of \emph{general
position} of Minkowski sums of simplices. For example, if all the simplices
are in \emph{relative general position} in the sense of~\cite{AS2016}, then
their colorful Gale dual is in relative general position. We do not elaborate
further because we will only need Corollary~\ref{cor:dict_genericcentered} in
this direction.

We are interested in the face structure of $\Mink\I\A$.  We will call a face
$\U=\U_0\cup\cdots\cup\U_\s$ \Defn{mixed} if $\dim(A(\U_i))\geq 1$ for all
$1\leq i\leq \s$. A facet is \Defn{totally mixed} if $\U_i$ is a facet of
$\A_i$ for all $0\le i\le \s$.

Proposition~\ref{prop:dict} allows to easily translate many properties between
the two setups. In particular information on facets. Notice that, since facets
of $\A$ are inclusion-maximal faces, the corresponding sub-configuration of
$\G$ is an inclusion minimal subset containing the origin in its convex hull.
In particular, facets of $\Mink\I\A$ correspond to subsets
$\cS=\cS_0\cup\cdots\cup \cS_\s$ with $\cS_i\varsubsetneq \I_i$ such that
$\G(\cS)$ is affinely independent and $0\in \relint(\G(\cS))$.

\begin{lem}\label{lem:char_tmf}
Let $\I = \I_0 \cup \cdots \cup \I_d$ with $n_i = |\I_i| \ge 2$, and
let $\G:\I \rightarrow \R^d$ be a centered colorful configuration 
that is the colorful Gale transform of $\A : \I \to \R^{n - 2d -1}$. Let $\sigma=\{v_0,\dots,v_s\}$ be a hitting colorful simplex of $\G$ such that $0\in\Int\G(\sigma)$. Then $\bigcup \I_i\setminus\{v_i\}$ is a totally mixed facet of $\Mink\I\A$.

If moreover each of the $\A_i$ is affinely independent for
all $0 \le i \le d$, the converse holds: each totally mixed facet is of $\Mink\I\A$ is of the form
$\bigcup \I_i\setminus\{v_i\}$ for a hitting colorful simplex $\sigma=\{v_0,\dots,v_s\}$ with
$0\in\Int\G(\sigma)$.
\end{lem}

Observe that the reciprocal does not hold in the general case, since
such a totally mixed facet  $\Mink\U\A$ exists if and only if each $\A(\I_i)$
is a pyramid with basis $\A(\U_i)$. However, the case in which we are most
interested in is the Minkowski sum of simplices, in which every facet is a
pyramidal base.

In particular, in light of Corollary~\ref{cor:dict_genericcentered}, centered
colorful $\n$-configurations in relative general position can be interpreted
as the Cayley embedding of a collection of simplices $A_0,\dots,A_{d}$ in
$\R^D$, where $D=1+\sum (\dim(A_i)-1)$.  The dimensions are not arbitrary: The
only possibility that a generic collection of polytopes in $\R^D$ can have a
totally mixed facet is that $D=1+\sum (\dim(A_i)-1)$ (and then such a facet is
isomorphic to a product of facets of each $A_i$).

Then Theorem~\ref{thm:simplices} is just a result of translating Theorem~\ref{thm:CC_upper} to Minkowski sums.

\begin{proof}[Proof of Theorem~\ref{thm:simplices}]
    Let $\G$ be the colorful Gale transform of $\A$. By
    Lemma~\ref{lem:char_tmf}, the number of totally mixed facets of
    $\Mink\I\A$ is the number of hitting simplices of $\G$ that contain the
    origin in its relative interior. Since this is an open condition, the
    maximum is is clearly attained by a colorful configuration in general
    position, and the Corollary follows directly from
    Theorem~\ref{thm:CC_upper}.
\end{proof}

\section{Minkowski sums and generalized Gale transforms}
\label{sec:Ptransform}
\newcommand\FL{\mathcal{I}}%
\newcommand\F{\mathcal{F}}%

In this final section we wish to outline an alternative way to associate a
colorful configuration to a Minkowski sums. This is by way of
\emph{generalized} Gale transforms that were first described by
McMullen~\cite{McMullen1979}. He presented several applications in the study
of zonotopes and centrally symmetric polytopes, and one purpose of this
section is to show new applications of this construction.  This is best done
for polytopes and this section differs from the previous ones in that we focus
on polytopes and their vertex sets.

Let $P \subset \R^d$ be a full-dimensional polytope. We are interested in the
facial structure of $\pi(P)$, where $\pi : \R^d \rightarrow \R^e$ is a linear
projection with $e \le d$. We may assume that $0$ is in the interior of $P$
and thus there are linear forms $\ell_0,\dots,\ell_m : \R^d \to \R$ such that
\[
    P \ = \ \{ x \in \R^d : \ell(x) \le 1 \text{ for all } i = 0,\dots,m \}.
\]

For a face $F \subseteq P$, let $I(F) = \{ i \in [m] : \ell_i(x) = 1 \text{
for all } x \in F \}$, where $[m]=\{0,\dots,m\}$, and set 
\[
    \FL \ = \ \FL(P) \ := \ \{ I(F) : F \subseteq P \text{ face} \}.
\]
Since $I(F) \neq I(F')$ if and only if $F \neq F'$, $\FL$ ordered by inclusion
is isomorphic to (the opposite of) the face lattice of $P$ and we consider
this to be the \emph{combinatorial information} obtained from $P$. 

Now, if $\pi : \R^d \rightarrow \R^e$ is a linear projection, then its adjoint
is an injection $\pi^* : (\R^e)^\ast \hookrightarrow (\R^d)^*$ and let $L
\subseteq (\R^d)^\ast$ be its image. Finally, let $\phi : (\R^d)^\ast
\twoheadrightarrow (\R^d)^\ast / L \cong \R^{d-e}$ be the canonical
projection. We define the \Defn{$P$-transform} of $\pi$ as the point
configuration $\G_P : [m] \rightarrow \R^{d-e}$ by $G_P(i) = \phi(\ell_i)$ for
$i = 0,\dots,m$. This, in a strong way, depends on the \emph{geometry} of $P$.
The following lemma is the analog of Lemma~\ref{lem:gale} that has been
discovered in different contexts. See~\cite{Ziegler2004,SanyalZiegler2010} for
strengthenings.

\begin{lem}\label{lem:Pgale}
    Let $P \subset \R^d$ be a full-dimensional polytope. For $\pi : \R^d
    \rightarrow \R^e$, let $\G_P$ be the associated $P$-transform. For a
    proper face $F \subset P$ the following are equivalent
    \begin{enumerate}[\rm (i)]
        \item $F' = \pi(F)$    is a proper face of $\pi(P)$ and $\pi^{-1}(F')
            \cap P = F$,
        \item $ 0 \in \relint \conv \G_P(I(F)).$
    \end{enumerate}
\end{lem}
\begin{proof}
    Let $\ell' \in (\R^e)^\ast$ be a linear function such that $\ell'$ is
    maximized on $F'$ over $\pi(P)$. This is the case if $\ell =
    \pi^\ast(\ell')$ is maximized on $F$ over $P$. That is, if there are
    scalars $\lambda_i \ge 0$ for $i \in I(F)$ such that 
    \begin{equation}\label{eq:genGaledual}
      \ell = \sum_{i \in I(F)} \lambda_i \ell_i.
    \end{equation}
    Since $\ell$ is in the image of $\pi^\ast$, it follows that $\phi(\ell) =
    0$ and hence, after applying $\phi$ to \eqref{eq:genGaledual} and scaling
    by ${1}\slash{\sum{\lambda_i}}$ we get that $0 \in \conv \G_P(I(F))$. If
    there is no face $\wh F \supset F$ such that $\pi(\wh F) = F'$, then all
    the $\lambda_i$ must be strictly positive, and hence the origin  is in the
    relative interior of $\conv \G_P(I(F))$.

    For the converse, note that if there are $\lambda_i > 0$ for $i \in
    I(F)$ such that $\ell = \sum_{i \in I(F)} \lambda_i \ell_i$, then $F$ is
    the inclusion-maximal face that maximizes $\ell$. Moreover, if 
    $\ell$ is in the kernel of $\phi$, then, by definition, it is in the image
    of $\pi^\ast$ and hence $\ell = \pi^\ast(\ell')$ for some $\ell' \in
    (\R^e)^\ast$. That is, $\pi(F)$ maximizes $\ell'$ over $\pi(P)$.
\end{proof}

The following example shows the connection to usual Gale transforms.

\begin{example}\label{ex:simplexPtransform}
    For $n \geq  1$ fixed, let $\Delta \subset \R^{n-1}$ be the
    $(n-1)$-dimensional simplex with vertices $v_1,v_1,\dots,v_n$ $v_n= - (v_1
    + \dots + v_{n-1})$.  Its $n$ facet-defining linear inequalities
    $\ell_i(x) \le 1$ satisfy $\ell_i(v_j) = 1$ if $i \neq j$ and $<1$
    otherwise. If we write $F_J = \conv(v_i : i \in J )$ for the face
    corresponding to a subset $J \subseteq [n]$, then the convention above
    yields $I(F_J) = [n] \setminus J$.

    For a given polytope $Q = \conv(q_1,\dots,q_n) \subset \R^d$ on $n$
    vertices whose barycenter is the origin, there is a canonical surjective
    linear projection $\pi : \R^{n-1} \to \R^d$ that takes $v_i$ to $q_i$ for
    all $1 \le i \le n$. The corresponding $\Delta$-transform is then a point
    configuration $\G_{\Delta} : [n] \rightarrow (\R^{n-1})^\ast / \im(\pi^\ast) \cong
    \R^{n-1-d}$ and Lemma~\ref{lem:Pgale} yields that for any $J \subseteq
    [n]$ we have that
    \[
        \pi(F_J) \ = \ \conv(q_i : i \in J) \ \subseteq \ Q
    \]
    is a face of $Q$ if and only if $0 \in \relint \conv( \G_{\Delta}(i) : i
    \in [n] \setminus J )$.
\end{example}

This is exactly the statement of Lemma~\ref{lem:gale}.  Actually, for a
careful choice of the simplex $\Delta$, ordinary Gale
transforms are exactly $\Delta$-transforms.  In general $P$-transforms are a
very powerful tool to study polytopes under projections; cf.~\cite{RS12}. In
particular, the (well-)known \emph{cs transforms} for centrally-symmetric
polytopes and \emph{zonal transforms} for zonotopes are $P$-transforms for
crosspolytopes and cubes, respectively~\cite{McMullen1979}.

Of course, we may also fix the projection $\pi : \R^d \to \R^e $ and vary the
polytope $P \subset \R^d$. This approach directly allows us to associate a
colorful configuration to a Minkowski sum. For fixed $\s$ and $d$, the
\Defn{Minkowski projection} is the linear map $\mu : (\R^d)^{\s+1} \to \R^d$
given by
\[
    (x_0,\dots,x_\s) \ \mapsto \ x_0 + \cdots + x_\s.
\]
For polytopes $P_0,\dots,P_\s \subset \R^d$ it is obvious that $\mu(P_0 \times
\cdots \times P_\s) = P_0 + \cdots + P_d$. If $P_i$ has facet-defining linear
forms indexed by $\F_i$, then the facet-defining linear forms of $P = P_0 +
\cdots + P_\s$ are indexed by $\F = \F_0 \cup \cdots \cup \F_\s$. The
\Defn{Minkowski transform} $M : \F \to \R^{(\s-1)d}$ of $P_0,\dots,P_\s$ is
then the $(P_0 \times \cdots \times P_\s)$-transform for the Minkowski projection $\mu$. 
This is a colorful configuration. It is also centered, because the normal vectors 
to the facets of each polytope are positively spanning in their linear span.
For faces $F_j \subseteq P_j$, $0 \le j \le \s$, write $I(F_j)
\subseteq \F_j$.  Then Lemma~\ref{lem:Pgale} yields that $F_0 +  \cdots + F_\s$
is a proper face of $P_0 + \cdots + P_\s$ if and only if 
\[
    0 \in \relint \conv M(I_0(F_0) \cup \cdots \cup I_\s(F_\s)).
\] 

Note that $M$ is a colorful configuration in $(s+1)d$-dimensional space and
thus typically of much lower dimension than the colorful Gale transform of
Section~\ref{sec:cayley}. However, as we pointed out in
Section~\ref{sec:cayley}, the colorful Gale transform also contains
information about subsums $\sum_{i \in I} P_i$ for any $I\subseteq[\s]$, and in particular, of each of the individual summands. This is not true any more
for the Minkowski transform $M$. In general, the number of points in the two transforms 
does not coincide either: The Minkowski transform is indexed by facets and the colorful
Gale transform by vertices.

Yet, these two approaches are strongly related.  Consider the case where $P_i
\subset \R^d$ is an $(n_i-1)$-dimensional simplex for all $0 \le i \le \s$.
If $\I_i$ and $\F_i$ index the facets and the vertices of $P_i$, respectively,
then $|\I_i|=|\F_i| = n_i$ and we can identify both by matching each facet
with the only vertex it does not contain.  Then, as in
Example~\ref{ex:simplexPtransform}, for $J_i\subset\I_i$ the face
$F_{J_i}=\conv(q_j\ \colon\ j\in J_i)$ belongs to the facets indexed by
$I(F_j)=\F_i\setminus J_i$.

The Minkowski transform $M : \F \rightarrow \R^{n - s - d - 1}$, where $n =
n_0 + \cdots + n_\s$ is a centered colorful configuration, because the normal
vectors to the facets of each polytope are positively spanning in their linear
span.  Lemma~\ref{lem:Pgale} yields that for any collection $J_i \subseteq
\I_i$ for $0 \le i \le \s$, the sum $F_{J_0} + \cdots + F_{J_\s}$ is a face of
$P_0 + \cdots + P_\s$ if and only if
\[
    0 \ \in \ \conv( M(\F_0\setminus J_0) \cup \cdots \cup M(\F_\s\setminus
    J_\s) ).
\]
Notice how this is exactly the condition of Proposition~\ref{prop:dict}.  This
shows that the Minkowski transformation $M$ and the colorful Gale transform
$\G$ of $P_0,\dots,P_\s$ are both colorful configurations in $(n - s - d -
1)$-dimensional space with the same combinatorics. A careful computation shows
that these actually coincide.

\begin{prop}
    For any collection of simplices $P_0,\dots,P_\s \subset \R^d$, the
    Minkowski transformation and the colorful Gale transformation coincide up
    to a choice of coordinates.
\end{prop}

\bibliographystyle{siam}
\bibliography{ColorfulSimplicialDepth}

\end{document}